\documentclass[oneside,english]{amsart}
\usepackage[T1]{fontenc}
\usepackage[latin9]{inputenc}
\usepackage{geometry}
\usepackage{amstext}
\usepackage{amsthm}
\usepackage{amssymb}
\usepackage{graphicx}
\usepackage{hyperref}
\usepackage{bbm}

\makeatletter
\numberwithin{equation}{section}
\numberwithin{figure}{section}
\theoremstyle{plain}
\newtheorem{thm}{\protect\theoremname}
  \theoremstyle{remark}
  
  \theoremstyle{definition}
  
  \theoremstyle{plain}
  \newtheorem{cor}[thm]{\protect\corollaryname}
  \newtheorem{lem}[thm]{\protect\lemmaname}
  \newtheorem{defn}[thm]{\protect\defname}
  \newtheorem{conj}[thm]{\protect\conjname}
\makeatother

\usepackage{babel}
  \providecommand{\corollaryname}{Corollary}
  \providecommand{\examplename}{Example}
  \providecommand{\remarkname}{Remark}
  \providecommand{\defname}{Definition}
  \providecommand{\conjname}{Conjecture}
  \providecommand{\lemmaname}{Lemma}
\providecommand{\theoremname}{Theorem}

\begin{document}

\title{Romik's Conjecture for the Jacobi Theta Function}

\author{Tanay Wakhare$^{\dag}$}
\address{$^{\dag}$~University of Maryland, College Park, MD 20742, USA}
\email{twakhare@gmail.com}

\maketitle

%{\Large

\begin{abstract}
Dan Romik recently considered the Taylor coefficients of the Jacobi theta function around the complex multiplication point $i$. He then conjectured that the Taylor coefficients $d(n)$ either vanish or are periodic modulo any prime ${p}$; this was proved by the combined efforts of Scherer and Guerzhoy-Mertens-Rolen, who considered arbitrary half integral weight modular forms. We refine previous work for $p \equiv 1 \pmod{4}$ by displaying a concise algebraic relation between $d\left( n+ \frac{p-1}{2} \right)$ and $d(n)$ related to the $p$-adic factorial, from which we can deduce periodicity with an effective period.
\end{abstract}

%TODO: Dedicated to Benjamin Meit, 1997-2016. You'd have liked this.

\section{Introduction}
Last year, Dan Romik published a fundamental paper \cite{Romik} considering the Taylor expansion of the classical \textit{Jacobi theta function} $\theta_3$, defined as 
$$\theta_3(x):= \sum_{n=-\infty}^\infty e^{-\pi n^2 x}.$$
It satisfies an important modular transformation given by
$$ \theta_3\left(\frac{1}{x}\right) = \sqrt{x} \theta_3(x).$$
The fixed point of this transformation, $x=1$, is the natural point to Taylor expand around. In fact, the natural function to study is $$\sigma_3(z):=\frac{1}{\sqrt{1+z}} \theta_3\left( \frac{1-z}{1+z} \right) = \theta_3(1) \sum_{n=0}^\infty \frac{d(n)}{(2n)!} \left( \frac{\Gamma^8\left(\frac14\right)}{2^7\pi^4} \right)^n z^{2n}.$$
The reason for considering such a rescaled function is that the M\"obius transformation $x \mapsto z= \frac{1-x}{1+x}$ conformally maps the right half-place to the unit disc and sends the inversion map $x \mapsto \frac{1}{x}$ to the reflection $z \mapsto -z$, so that the modular transformation satisfied by $\theta_3$ is then equivalent to the statement $\sigma_3(z) = \sigma_3(-z)$, and Taylor expanding $\theta_3(x)$ around $x=1$ is equivalent to Taylor expanding the simpler $\sigma_3(z)$ around $z=0$.

%Romik found a recurrence for the Taylor coefficients $d(n)$, which however do not possess a closed form expression. 
Although Romik found a recurrence for the Taylor coefficients $d(n)$ in Definition \ref{dndef}, there does not seem to exist a closed form expression for these coefficients. They depend on a sequence $s(n,k)$, which in turn depends on a sequence $u(n)$, which is given by a recurrence relation. This triply nested definition makes it rather unwieldy to work with the Taylor coefficients directly, though Romik conjectured several nice properties of the $d(n)$ coefficients modulo any prime. This paper is dedicated towards refining the second half of Romik's conjecture, which was proven by the combined efforts of Scherer \cite{Scherer} and Guerzhoy-Mertens-Rolen \cite{Rolen}. Guerzhoy, Mertens, and Rolen in fact prove a stronger statement in the context of an arbitrary half integer weight modular form.
\begin{thm}\cite{Scherer, Rolen}
Modulo any prime $p$, $d(n)$ exhibits well defined behavior.
\begin{itemize}
\item For any prime $p \equiv 3 \pmod{4}$, there exists an $n_0$ such that $d(n)\equiv 0 \pmod{p}$ for $n> n_0$.
\item For any prime $p \equiv 1 \pmod{4}$ or $p=2$, $d(n) \pmod{p}$ is periodic.
\end{itemize}

\end{thm}

We can now prove a stronger result in the case $p \equiv 1 \pmod{4}$. 
\begin{thm}\label{mainthm}
Consider $ p \equiv 1 \pmod{4}$. Then for $n \geq \frac{p+1}{2}$, we have
$$d\left( n+ \frac{p-1}{2} \right) \equiv -2^{\frac{p-1}{2}}3^27^2\cdots (2p-3)^2  d(n)   \pmod{p} . $$
\end{thm}

This encodes more arithmetic information, and on iterating this $p-1$ times we see that $d(n)$ is periodic with period $\frac{(p-1)^2}{2}$, which is not necessarily minimal.  The proof structure is tripartite:
\begin{enumerate}
\item show that $u(n), v(n) \equiv 0 \pmod{p}$ for $n\geq \frac{p+1}{2}$, a result conjectured in Scherer's paper \cite{Scherer};
\item use this to simplify the expression for $s(n,k) \pmod{p}$;
\item use the expression for $s(n,k)$ to show the desired periodicity of $d(n) \pmod{p}$.
\end{enumerate}
Our methods are elementary, and consist of a tour through classical number theory and group theory. We essentially used to method of \cite{Scherer}, who studied the $p=5$ case; however, the proof for arbitrary $p$ becomes \textit{significantly} more technically complex. The modular form approach of Guerzhoy-Mertens-Rolen \cite{Rolen} is extremely beautiful, since it proves eventual periodicity for any weight $1/2$ modular form. However, with our method we can show not just periodicity but also a finer algebraic relation between $d\left(n+ \frac{p-1}{2}\right)$ and $d(n)$.

%TODO: serre citation

The study of the Fourier coefficients of modular forms is a prominent thread of twentieth century mathematics; this gives us a first hint at a similarly deep theory, where Taylor coefficients of various modular forms expanded around complex multiplication points have $p$-adic properties analogous to the Fourier coefficients. A deep theorem of Ono and Skinner \cite{Ono} states that for a ``good" half integral weight modular form, for all but finitely many primes $\ell$, there are infinitely many Fourier coefficients at squarefree indices divisible by $\ell$. We expect similarly sweeping results about the Taylor coefficients of half integral weight modular forms to hold as well. It would also be nice to see the Taylor coefficients of other fundamental modular forms studied explicitly, such as the elliptic $j$ invariant or Dedekind $\eta$--function.

%The next natural step is to consider the Taylor coefficients of $\eta(z)$, the Dedekind eta function, and its various ratios, since these generate many modular forms. Alternatively, we can consider the expansion of the elliptic $j$ invariant $$j = 1728 \frac{g_2(\tau)^3}{g_2(\tau)^2-27g_3(\tau)^2}, $$
%where 
%$$g_2(\tau):= 60 \sum_{(m,n)\neq (0,0)} \frac{1}{(m+n\tau)^4}, \medspace\medspace g_3(\tau):= 140 \sum_{(m,n)\neq (0,0)} \frac{1}{(m+n\tau)^6}$$
%are the first Eisenstein series, since the fraction field $\mathbb{C}(j, j^{-1})$ generates the ring of arbitrary weight modular forms.

\section{Vanishing of $u(n)$}
We begin with the definition of the $u(n)$ and $v(n)$ coefficients, in terms of hypergeometric functions and a recurrence. We require the Gauss hypergeometric series
$${}_{2}F_{1}\left(\begin{array}{c}
a, b\\
c
\end{array};z\right).:= \sum_{n=0}^\infty \frac{(a)_n(b)_n}{n!(c)_n}z^n,$$
with $(a)_n:= \prod_{i=0}^{n-1}(a+i)$ denoting a Pochhammer symbol.
\begin{defn}\cite{Romik}
We define our coefficients in terms of the following generating functions:
$$\sum_{n=0}^\infty \frac{u(n)}{(2n+1)!}t^n  := \frac{  {}_{2}F_{1}\left(\begin{array}{c}
\frac34, \frac34\\
\frac32
\end{array};4t\right) }{  {}_{2}F_{1}\left(\begin{array}{c}
\frac14,\frac14\\
\frac12
\end{array};4t\right).}$$
and
$$\sum_{n=0}^\infty \frac{v(n)}{2^{n}(2n)!}t^n  := \sqrt{{}_{2}F_{1}\left(\begin{array}{c}
\frac14, \frac14\\
\frac12
\end{array};4t\right).}$$
Equivalently, we set $u(0)=v(0)=1$ and calculate them using the recurrences
\begin{equation}\label{urecur}
u(n) = 3^27^2\cdots (4n-1)^2 - \sum_{m=0}^{n-1}\binom{2n+1}{2m+1}1^25^2 \cdots (4(n-m)-3)^2 u(m)
\end{equation}
and
\begin{equation}\label{vrecur}
v(n) =2^{n-1} 1^25^2\cdots (4n-3)^2 - \frac12 \sum_{m=1}^{n-1}\binom{2n}{2m}v(m)v(n-m).
\end{equation}
\end{defn}

\begin{defn}\label{dndef}\cite{Romik}
We define an auxiliary matrix of coefficients by
$$ s(n,k) = \frac{(2n)!}{(2k)!} [z^{2n}]\left(  \sum_{j=0}^\infty \frac{u(j)}{(2j+1)!}z^{2j+1} \right)^{2k}.$$
Finally, we can define the sequence
$$ d(n) = v(n) - \sum_{k=1}^{n-1} 2^{n-k} s(n,k) d(k).$$
\end{defn} 
As a corollary of the integrality of $s(n,k)$, the $d(n)$ coefficients are also integers.

For the proof of the Lemma \ref{ulem}, we will heavily depend on a classical theorem of Lucas regarding the congruence properties of binomial coefficients:
\begin{lem}(Lucas's Theorem)\label{kummer}
For any prime $p$, write $n$ and $k$ in terms of their base $p$ expansions, denoted by $n = (n_ln_{l-1}\ldots n_0)_p$ and $k = (k_lk_{l-1}\ldots k_0)_p$. Then
$$\binom{n}{k} \equiv \prod_{i=0}^l \binom{n_i}{k_i} \pmod{p},$$
where we set $\binom{n}{k}=0$ if $k > n$.
\end{lem}

We can then show the following fundamental result, which was conjectured by Scherer\cite{Scherer}. The only tools we need are Lucas's Theorem and a careful multistep induction.
\begin{lem}\label{ulem}
Consider $p \equiv 1 \pmod{4}$. Then $u(n) \equiv v(n) \equiv 0 \pmod{p}$ for $n \geq \frac{p+1}{2}$.
\end{lem}
\begin{proof}
Consider $n= \frac{p-1}{2}$, so that $$\binom{2n+1}{2m+1} = \binom{p}{2m+1} \equiv \binom{1}{(2m+1)_1}\binom{0}{(2m+1)_0} \pmod{p}$$ by Lucas's Theorem. However, for $0\leq m \leq n-1$, $2m+1$ always has a nonzero last digit, so that this binomial coefficient is always $\equiv 0$. Therefore, by recurrence \eqref{urecur} we have 
\begin{equation}\label{uendpoint}
u\left(\frac{p-1}{2}\right) \equiv   3^27^2\cdots (2p-3)^2 \pmod{p}.
\end{equation}
Now take $n = \frac{p+1}{2}$, so that $2n+1 = p+2= (12)_p$. The only values of $2m+1$ for $0\leq m \leq n-1$ with last digit $<2$ correspond to $m=0$ and $m = \frac{p-1}{2}$. Thus, by Lucas's Theorem, we have
\begin{align*}
u\left(\frac{p+1}{2}\right) \equiv   3^27^2\cdots (2p+1)^2  - u(0) \binom{p+2}{1} 1^25^2\cdots (2p-1)^2 -u\left(\frac{p-1}{2}\right)  \binom{p+2}{p} \pmod{p}.
\end{align*}
Now we use Lucas's Theorem to write $\binom{p+2}{p} \equiv \binom{1}{1}\binom{2}{0} \equiv 1 \pmod{p}$ and note that since $p \equiv 1 \pmod{4}$, $p$ divides one of the terms in the product corresponding to the $u(0)$ term. Therefore, using our previous calculation for $u\left(\frac{p-1}{2}\right) $,
\begin{align*}
u\left(\frac{p+1}{2}\right)  \equiv   3^27^2\cdots (2p-3)^2(2p+1)^2  - 3^27^2\cdots (2p-3)^2 \equiv 0 \pmod{p},
\end{align*}
since $(2p+1)^2 \equiv 1$. This provides the base case we now require for an inductive proof. Assume $u\left({\frac{p+1}{2} + l}\right) \equiv 0$ for $0 \leq l < N$, for some $N \leq \frac{p-3}{2}$. Now consider $u\left({\frac{p+1}{2}+N}\right)$, with the bound on $N$ ensuring that $2\left(\frac{p+1}{2}+N\right)+1 \leq 2p-1$. We now decompose recurrence \eqref{urecur} as follows, and show that the contribution from each line separately totals $0$:
\begin{align}
u\left({\frac{p+1}{2}+N}\right)= 3^27^2\cdots &(4N+2p+1)^2 - \binom{p+2+2N}{p} 1^25^2\cdots (4N+1)^2 u\left({\frac{p-1}{2}}\right) \label{line1} \\
&- \sum_{m=\frac{p+1}{2}}^{\frac{p+1}{2}+N-1}\binom{p+2+2N}{2m+1}1^25^2 \cdots (2p-1+4N-4m)^2 u(m) \label{line2}\\
&- \sum_{m=N+1}^{\frac{p-3}{2}}\binom{p+2+2N}{2m+1} 1^25^2 \cdots (2p-1+4N-4m)^2 u(m) \label{line3}\\
&- \sum_{m=0}^{N}\binom{p+2+2N}{2m+1} 1^25^2 \cdots (2p-1+4N-4m)^2 u(m) \label{line4}.
\end{align}
First, consider the contribution from line \eqref{line1}, which isolates the $m= \frac{p-1}{2}$ term and the constant term. By the bound on $N$, the first digit in the base $p$ expansion of $2\left(\frac{p+1}{2}+N\right)+1$ will be a $1$, and the second digit will be $2(N+1)$, which is less than $p$.  By Lucas's theorem we have
$$\binom{p+2+2N}{p} \equiv \binom{1}{1} \binom{2+2N}{0} \equiv 1 \pmod{p}.$$
Also note that we previously calculated $u\left({\frac{p-1}{2}}\right) \equiv   3^27^2\cdots (2p-3)^2$, so that
\begin{align*}
1^25^2\cdots (4N+1)^2 u\left({\frac{p-1}{2}}\right)& \equiv   3^27^2\cdots (2p-3)^2 \times 1^25^2\cdots (4N+1)^2  \\
&\equiv  3^27^2\cdots (2p-3)^2 \times (2p+1)^2(2p+5)^2\cdots (4N+2p+1)^2 \pmod{p},
\end{align*}
which exactly cancels with the constant term in the $u(n)$ recurrence! Therefore, line \eqref{line1} contributes nothing -- the most delicate part of the proof. Line \eqref{line2} vanishes modulo $p$ by the inductive hypothesis. Line \eqref{line3} vanishes by Lucas's Theorem, since 
$$ \binom{p+2+2N}{2m+1} \equiv \binom{1}{0}\binom{2+2N}{2m+1} \pmod{p},$$
but $2+2N < 2m+1 <p$ in this range of summation, hence the binomial is congruent to $0$. Line \eqref{line4} vanishes since  $2p-1+4N-4m \geq 2p+1$ in this region of summation, so that since $p\equiv 1 \pmod{4}$, the product actually contains $p^2$ and vanishes modulo $p$.

Therefore, we've shown $u(n) \equiv 0 \pmod{p}$ for $\frac{p+1}{2}\leq n \leq p-1.$ We could not prove the general case all at once since the previous proof depended on Lucas's Theorem, which is sensitive to the last digit of our index $n$. The rest of the induction however follows easily, since we rewrite recurrence \eqref{urecur} as 
\begin{align}
u(N) = 3^27^2\cdots (4N-1)^2 -&\sum_{m=0}^{\frac{p-1}{2}}\binom{2N+1}{2m+1}1^25^2 \cdots (4(N-m)-3)^2 u(m) \label{line1a}\\
-&\sum_{m=\frac{p+1}{2}}^{N-1}\binom{2N+1}{2m+1}1^25^2 \cdots (4(N-m)-3)^2 u(m) \label{line1b}.
\end{align}
where $N\geq p$. We inductively assume that $u(n)=0$ for $\frac{p+1}{2}\leq n < N$. Since $N \geq p$ and $p \equiv 1 \pmod{4}$, we have $3p<4N-1$ and $3p \equiv 3 \pmod{4}$, so that $(3p)^2$ divides the constant term in line \eqref{line1a}, which must vanish mod $p$. Consider the sum in line \eqref{line1a}; since $N \geq p$ and $m \leq \frac{p+1}{2}$, we have $4(N-m)-3 \geq 4 (\frac{p+1}{2})-3 = 2p+1 $. Since $p \equiv 1 \pmod{4}$, $p^2$ must always divide the product in the sum and line \eqref{line1a} completely vanishes. Line \eqref{line1b} vanishes by the inductive hypothesis, completing the proof.

To show $v(n)\equiv 0$ for $n\geq \frac{p+1}{2}$, we adopt a similar proof structure; we will prove it for $\frac{p+1}{2}\leq n<p$ using Lucas's Theorem, and then show it for $n \geq p$ by a separate induction. First, consider $n = \frac{p+1}{2}$, so that $2n = (11)_p$. Then for $1\leq m \leq n-1$, $2m$ always has last digit $\geq 2$. Hence by Lucas's Theorem, the binomial coefficient always vanishes mod $p$. We then note that $p < 4n-3$ so that since $p\equiv 1 \pmod{4}$, $p^2$ divides the constant term. Hence $v\left({\frac{p+1}{2}}\right) \equiv 0$. 

We now proceed by induction. Assume $v\left({\frac{p+1}{2} + l}\right) \equiv 0$ for $0 \leq l < N$, for some $N \leq \frac{p-3}{2}$. Now consider $v\left({\frac{p+1}{2}+N}\right)$, with the bound on $N$ ensuring that $2\left(\frac{p+1}{2}+N\right)\leq 2p-2$. We now decompose the recurrence \eqref{vrecur} as follows:
\begin{align}
v\left(\frac{p+1}{2}+N\right)  &=2^{n-1} 1^25^2\cdots (2p-1+4N)^2  \label{line3a}\\
&- \frac12 \sum_{m=1}^{N}\binom{p+1+2N}{2m}v(m)v\left(\frac{p+1}{2}+N-m\right) \label{line3b}\\
&- \frac12 \sum_{m=\frac{p+1}{2}}^{\frac{p+1}{2}+N}\binom{p+1+2N}{2m}v(m)v\left(\frac{p+1}{2}+N-m\right)  \label{line3c} \\
&- \frac12 \sum_{m=N+1}^{\frac{p-1}{2}}\binom{p+1+2N}{2m}v(m)v\left(\frac{p+1}{2}+N-m\right)  \label{line3d}.
\end{align}
We must have $p^2$ divide the constant term, so line \eqref{line3a} vanishes. By the induction hypothesis, the $v\left(\frac{p+1}{2}+N-m\right) $ term in line \eqref{line3b} vanishes, as does the $v(m)$ term in line \eqref{line3c} vanishes. Since $2\left(\frac{p+1}{2}+N\right) = p +1 + 2N$ has first digit $1$ and last digit $1+2N$, the binomial coefficient in line \eqref{line3d} evaluates to (noting $2N-2 \leq p-1 $ has one digit when expanded in base $p$)
$$ \binom{p+1+2N}{2m} \equiv \binom{1}{0}\binom{1+2N}{2m} \equiv 1 \cdot 0 \pmod{p}. $$
Hence $v\left({\frac{p+1}{2}+l}\right) \equiv 0$ for $ 0\leq l \leq \frac{p-3}{2}$ and thus $v(n) \equiv 0$ for $\frac{p+1}{2} \leq n \leq p-1$, completing the induction in this regime.

We now complete the induction for all $n$. Assume $N \geq p$ and $v(n)\equiv 0$ for all $\frac{p+1}{2} \leq n \leq N-1$. 
Then we write \eqref{vrecur} as 
\begin{align}
v\left(N\right)  &=2^{n-1} 1^25^2\cdots (4N-3)^2  \label{line4a}\\
&- \frac12 \sum_{m=\frac{p+1}{2}}^{N-1}\binom{2N}{2m}v(m)v\left( N-m\right) \label{line4b}\\
&- \frac12\sum_{m=1}^{\frac{p-1}{2}}\binom{2N}{2m}v(m)v\left( N-m\right) \label{line4c}.
\end{align}
Note $p \equiv 1 \pmod{4}$ so $p^2$ divides the term in line \eqref{line4a}, the $v(m)$ term in line \eqref{line4b} vanishes by the induction hypothesis, and the $v\left( N-m\right)$ term in line \eqref{line4c} vanishes by the induction hypothesis. Therefore $v(N) \equiv 0 \pmod{p}$ and we're done.
\end{proof}

\section{Reduction of $s(n,k)$}
Now that we have good control over the behavior of the $u$ and $v$ coefficients, we can reduce the $s(n,k)$ coefficients. Recall the definition
$$ s(n,k) = \frac{(2n)!}{(2k)!} [z^{2n}]\left(  \sum_{j=0}^\infty \frac{u(j)}{(2j+1)!}z^{2j+1} \right)^{2k}, $$
where $ [z^{2n}]$ denotes the coefficient of $z^{2n}$. We can then do the obvious thing and expand the product, collecting coefficients of each unique multinomial $u(j_1)^{c_1}u(j_2)^{c_2}\cdots$. This was done by Scherer, but before presenting his result we need to introduce some notation.

Given an integer $n$, a \textit{partition} of $n$ is a tuple $\lambda = (\lambda_1, \ldots, \lambda_k)$, arranged in weakly decreasing order (so that $\lambda_i \geq \lambda_{i+1}$), such that $\sum_{i=1}^k \lambda_k = n$. Every partition can equivalently be described by a tuple $(c_1,\ldots, c_n)$, where $c_i$ denotes the number of times the part $i$ appears in the partition $\lambda$. For instance, consider the partition $(3,1,1,1)$ of $6$. This can be described by the tuple $(c_1=3, c_2=0, c_3=1, c_4=0, c_5=0, c_6=0)$. Let $\mathfrak{P}_{n,k}$ denote the set of partitions of $n$ into $k$ \textit{odd} parts. Given $\lambda \in \mathfrak{P}_{n,k}$, consider the associated tuple $(c_1,\ldots,c_n)$. Then 
$$ N_\lambda := \frac{n!}{\prod_{i=1}^n i!^{c_i} c_i!} $$
is an integer \cite{Andrews} counting the number of set partitions of $N$ elements into $k$ blocks $\{B_i\}_{i=1}^k$ with $|B_i| = \lambda_i$. 
\begin{lem}\label{slem} \cite[Thm 8]{Scherer}
We have the alternate expansion
$$s(n,k) = \sum_{\lambda \in \mathfrak{P}_{2n,2k}} N_\lambda \prod_{i=1}^{2n} u\left( \frac{i-1}{2} \right)^{c_i}.$$
\end{lem}
Note that since $u(n), N_\lambda$ are always integers, we also have that $s(n,k)$ is an integer. The purpose of Lemma \ref{ulem} was to eliminate many of the partitions in $\mathfrak{P}_{2n,2k}$ from this sum, since if $\lambda$ contains a part $i > p$, then $u\left( \frac{i-1}{2} \right) \equiv 0$ and the corresponding term vanishes. Therefore, we only need to consider $\lambda$ with largest part $p$ in Lemma \ref{mainlem}. In fact, when we reduce mod $p$ we can show that $N_\lambda \equiv 0$ \textit{almost always}, except for a small set of partitions which we can explicitly characterize. In the case $p=5$, the (rather technical) Lemma \ref{mainlem} reduces to \cite[Thm 11]{Scherer}.

Our restrictions on $\lambda \in \mathfrak{P}_{2n,2k}$ translate into the following set of equations in $\frac{p+1}{2}$ unknowns $\{c_1,c_3,c_5,\ldots, c_p\}$:
\begin{align*}
c_1+c_3+\cdots + c_p &= 2k, \\
c_1+3c_3+5c_5+\cdots + pc_p &= 2n,
\end{align*}
where $c_i$ denotes the number of occurrences of the part $i$. Throughout this section, we will use the shorthand
$$\alpha:= c_3 + c_5 + \cdots + c_{p-2}, \medspace\medspace \beta:=  3c_3 +5c_5+\cdots +(p-2)c_{p-2}, $$
so that our defining equations reduce to
\begin{align}
c_1+ \alpha + c_p &= 2k, \label{ceq1}\\
c_1+\beta+ pc_p &= 2n, \label{ceq2}
\end{align}

We also require a classical theorem of Legendre:
\begin{lem}(Legendre's Formula)
Let $\nu_p(n)$ denote the largest power of $p$ dividing $n$. Then 
$$\nu_p(n) = \frac{n-s_p(n)}{p-1} = \sum_{i=1}^\infty \left \lfloor{ \frac{n}{p^i}}\right \rfloor, $$
where $s_p(n)$ is the sum of the digits of $n$ when expressed in base $p$.
\end{lem}

\begin{lem}\label{mainlem}
Let $p\equiv 1 \pmod{4}$. For $\lambda \in \mathfrak{P}_{2n,2k}$ with largest part $p$, $N_\lambda \equiv 0 \pmod{p}$ unless 
\begin{itemize}
\item we have $\alpha\leq \beta <p$;
\item we have $c_p =\left \lfloor{ \frac{2n-2k}{p-1}}\right \rfloor $, which is the largest possible value of $c_p$.
\end{itemize}

\end{lem}
\begin{proof}
We will find a series of successively stronger restrictions on the parts $c_i$ until we are forced to reach the lemma's conclusion.
\begin{enumerate}
\item We want to show that for $i=3,5,\ldots,p-2$, we must have $c_i \leq \left \lfloor{ \frac{p}{i}}\right \rfloor <p$. We begin by noting that $\nu_p(n!) = 0$ for $n<p$, and using Legendre's formula to write
\begin{align*}
\nu_p(N_\lambda) &= \nu_p\left(  \frac{(2n)!}{p^{c_p}\prod_{i=1}^p c_i!}  \right) \\
&= \frac{2n - s_p(2n)}{p-1} - c_p - \sum_{i=1,3,\ldots, p} \frac{c_i - s_p(c_i)}{p-1}.
\end{align*}
Hence 
\begin{align*}
(p-1)\nu_p(N_\lambda) &= {2n - s_p(2n)} - c_p(p-1) - \sum_{i=1,3,\ldots, p} ({c_i - s_p(c_i)}) \\
&= {2n - s_p(2n)} - \sum_{i=1,3,\ldots,p }c_i(i-s_p(i)) - \sum_{i=1,3,\ldots, p} ({c_i - s_p(c_i)}) \\
&= -s_p(2n) +\sum_{i=1,3,\ldots,p}c_i(s_p(i)-1) +\sum_{i=1,3,\ldots,p}s_p(c_i).
\end{align*}
Now note that $s_p(i) = i$ for $i<p$ so the middle sum simplifies, while $s_p(1)=s_p(p)=1$ so these two terms completely drop out. Now we use the inequality $$s_p\left( \sum_{i}a_i\right) \leq \sum_i s_p(a_i),$$ since $$\sum_i s_p(a_i)  - s_p\left( \sum_{i}a_i\right) = (p-1)\nu_p \binom{\sum_i a_i}{a_1,a_2,\ldots,a_n} \geq 0$$ is the multinomial generalization of Legendre's formula. Therefore, 
\begin{align*}
(p-1)\nu_p(N_\lambda) &=  -s_p\left(c_1+3c_3+\cdots+pc_p\right) +\sum_{i=3,5,\ldots,p-2}c_i(i-1) +\sum_{i=1,3,\ldots,p}s_p(c_i) \\
&\geq -\sum_{i=1,3,\ldots,p} s_p(ic_i)+\sum_{i=1,3,\ldots,p}s_p(c_i) +\sum_{i=3,5,\ldots,p-2}c_i(i-1) \\
&= \sum_{i=3,5,\ldots,p-2}c_i(i-1) +s_p(c_i) - s_p(ic_i) \\
&= (p-1)\sum_{i=3,5,\ldots,p-2} \nu_p(ic_i!) - \nu_p(c_i!).
\end{align*}
For $i=3,5,\ldots,c_{p-2}$, once $c_i = \left \lceil{ \frac{p}{i}}\right \rceil$, then $\nu_p(ic_i!) - \nu_p(c_i!)\geq 1$. Now since $i \geq 3$, any larger values of $c_i$ will also lead to the lower bound $\nu_p(ic_i!) - \nu_p(c_i!)\geq 1$. Thus in this case $\nu_p(N_\lambda)>0$ and hence $N_\lambda \equiv 0 \pmod{p}$.

\item We want to show $\beta < p$. First we require an auxiliary lemma based on Legendre's formula:
\begin{align*}
\nu_p(p^nn!) = n + \nu_p(n!) = n+ \sum_{i=1}^\infty \left \lfloor{ \frac{n}{p^i}}\right \rfloor =  \sum_{i=1}^\infty \left \lfloor{ \frac{np}{p^i}}\right \rfloor = \nu_p((np)!). 
\end{align*}
Now based on the previous bound on $c_i$, we can freely assume that $c_i < p, 3\leq i \leq p-2$, for the rest of the proof. Therefore, we know $\nu_p(c_i!)=0$ for these values of $i$. Now note that we can write $2n-\beta = c_1 + pc_p$ by the definition of $\beta$, so that we can apply our previous lemma to write
\begin{align*}
\nu_p(N_\lambda) &= \nu_p\left(  \frac{(2n)!}{p^{c_p}c_p! c_1!}  \right) = \nu_p\left(  \frac{(2n)!}{(pc_p)! c_1!}  \right) \\
&=  \nu_p\left(  \frac{(2n)!}{(2n-\beta)!}\binom{2n-\beta}{pc_p, c_1}  \right) \geq  \nu_p\left(  \frac{(2n)!}{(2n-\beta)!}  \right).
\end{align*}
Now if $\beta \geq p$ then the set of $\geq p$ consecutive integers $\{2n, 2n-1, \ldots, 2n-\beta+1\}$ will contain a member divisible by $p$, so that we can lower bound the valuation by $1$, showing that $N_\lambda \equiv 0 \pmod{p}$ in this case. Hence we must have $\beta < p$ in order to have a zero $p$-adic valuation.

\item We now show the conclusion of the lemma. Consider Equations \eqref{ceq1} and \eqref{ceq2}. We trivially have $p> \beta \geq \alpha \geq 0$. The key insight here is that we can strengthen the previous inequality to $0 \leq \beta-\alpha<p-1$. If $\beta= 0$, this forces $\alpha = 0$ and we're done. If $\beta >0$ we necessarily have $\alpha>0$ so that $\beta - \alpha < p - 1$. Subtracting the two defining equations gives
$$(p-1)c_p +\beta-\alpha = 2n-2k, $$
but since $2n-2k$ has a unique base $p-1$ representation and $\beta-\alpha<p-1$ must necessarily be the remainder $2n-2k \pmod{p-1}$, there is a \textit{unique} value of $c_p$ we care about, which is precisely $\left \lfloor{ \frac{2n-2k}{p-1}}\right \rfloor$.
\end{enumerate}
\end{proof}
What this lemma essentially says is that if $c_p$ does not dominate $\lambda$, in that it's not as large as possible, then $N_\lambda$ will vanish mod $p$. Meanwhile, the contribution from $c_1$ can be unbounded. For example, only the single partition with $c_1 = 2n$ can contribute to the sum $s_{n,n}$.

We also note that the restriction $\beta:=3c_3+5c_5+\cdots+(p-2)c_{p-2} < p$ is rather strong; it is a Diophatine inequality in the $\frac{p-3}{2}$ variables $\{c_3,\ldots,c_{p-2}\}$, which has only finitely many solutions. Therefore, for a fixed prime $p$ we define the set of \textbf{core configurations} $\mathfrak{D}_p$ as the tuples $(c_3,\ldots,c_{p-2})$ satisfying $3c_3 + 5c_5+\cdots +(p-2)c_{p-2}<p$. For example, when $p=11$ there are ten possible core configurations: 
\begin{align*}
(c_3,c_5,c_7,c_{9}) \in \{   &(0,0,0,1), (0,0,1,0), (1,0,1,0), (0,2,0,0), (0,1,0,0), \\
&(1,1,0,0), (1,0,0,0), (2,0,0,0), (3,0,0,0), (0,0,0,0) \}. 
\end{align*}
For each $\mu:= (c_3,\ldots,c_{p-2}) \in \mathfrak{D}_p$ we can then define a norm function $$N(\mu) := 3c_3+5c_5+\cdots+ (p-2)c_{p-2}$$ and length function $$\ell(\mu):=c_3+c_5+\cdots+c_{p-2}.$$

%%%%%%%%%%%%%%%%%%%%%

We can then present a new (rather technical) decomposition of $s(n,k) \pmod{p}$ based on these core configurations. This is where our proof begins to differ significantly from Scherer's proof of the case $p=5$.
\begin{lem}\label{sdecomp}
Fix a prime $p$ satisfying $p\equiv 1 \pmod{4}$. Let $\mathfrak{D}_p$ denote the set of core configurations, and for each $\mu \in \mathfrak{D}_p$ associate the quantity 
$$N'_\mu = \prod_{i=1}^{\frac{p-3}{2}} \frac{u(i)^{c_{2i+1}}}{ (2i+1)!^{c_{2i+1}} c_{2i+1}!},$$
Also denote $$c_p^{(k)} := \frac{n-k - \left( \frac{N(\mu)-\ell(\mu)}{2} \right)}{  \frac{p-1}{2} }.$$ 
Then we have
\begin{align*}
s(n,k) \equiv \sum_{ \substack{ \mu \in \mathfrak{D}_p \\   \frac{N(\mu)-\ell(\mu)}{2}\equiv n-k  \mod{\frac{p-1}{2}}    }} &(2n)(2n-1) \cdots (2n - N(\mu)+1) N'_\mu \\
&\times \frac{ (2n-N(\mu))!  u\left( \frac{p-1}{2} \right)^{c_p^{(k)}} }{ (-1)^{c_p^{(k)}} p^{c_p^{(k)}}c_p^{(k)}! (2n-N(\mu) -pc_p^{(k)} )! } \pmod{p}.
\end{align*}
\end{lem}
\begin{proof}
Begin with Lemma \ref{slem}, that $$s(n,k) = \sum_{\lambda \in \mathfrak{P}_{2n,2k}} N_\lambda \prod_{i=1}^{2n} u\left( \frac{i-1}{2} \right)^{c_i}.$$
When we reduce modulo $p$, by the vanishing $u(n)\equiv 0 \pmod{p}, n\geq \frac{p+1}{2}$, we only sum over $\lambda \in \mathfrak{P}_{2n,2k}$ with largest part $\leq p$. Furthermore, by the first result of Lemma \ref{mainlem} we can further reduce the sum to $\lambda$ satisfying $3c_3+5c_5+\cdots+ (p-2)c_{p-2}<p$. The key insight is that we can then rearrange the sum over $\lambda$ based on the value of the tuple $(c_3,c_5,\ldots,c_{p-2})$. Every $\lambda = (c_1,c_3,\ldots, c_{p})$ has its sub-tuple $(c_3,\ldots,c_{p-2})$ fall into one of the \textit{finite} number of core configurations $\mathfrak{D}_p$. 

Furthermore, given that we are identifying $\lambda \in \mathfrak{P}_{2n,2k}$ with a Diophatine solution to 
\begin{align*}
c_1+c_3+\cdots + c_p &= 2k, \\
c_1+3c_3+5c_5+\cdots + pc_p &= 2n,
\end{align*}
any $\lambda$ with core configuration $\mu \in \mathfrak{D}_p$ satisfies
$$c_1 + \ell(\mu) + c_p = 2k,\medspace\medspace c_1 + N(\mu) +pc_p = 2n.$$
Holding the core configuration $\mu$ and integer $n$ fixed, this has solution 
$$c_p = \frac{2n-2k - \left( N(\mu) - \ell(\mu)\right) }{p-1}, \medspace\medspace c_1 = 2n-N(\mu) - pc_p.$$
Since $N(\mu) - \ell(\mu) = 2c_3 + 4c_5 + \cdots + (p-3)c_{p-2} \equiv 0 \pmod{2}$, we can in fact parametrize the possible integral solutions $c_p^{(k)}$ as
$$c_p^{(k)} = \frac{n-k - \left( \frac{N(\mu)-\ell(\mu)}{2} \right)}{  \frac{p-1}{2} }, $$
where $$  n-k \equiv  \left( \frac{N(\mu)-\ell(\mu)}{2} \right) \mod{\frac{p-1}{2}} . $$
Now we note that $u(0)=1$, so that $u(0)^{c_1}=1$ always. This allows us to account for the factor $\prod_{i=1}^{p}u\left( \frac{i-1}{2} \right)^{c_i}$ by ignoring $c_1$, folding the contribution from $\{c_3,\ldots,c_{p-2}\}$ into the definition of $N'_\mu$, and explicitly specifying the contribution from $u\left( \frac{p-1}{2}\right)$. Recalling the definition of $N'_\mu$, we can then sum over core configurations first to obtain
\begin{align*}
s(n,k) \equiv \sum_{ \substack{ \mu \in \mathfrak{D}_p \\   \frac{N(\mu)-\ell(\mu)}{2}\equiv n-k  \mod{\frac{p-1}{2}}    }} &(2n)(2n-1) \cdots (2n - N(\mu)+1) N'_\mu \\
&\times \frac{ (2n-N(\mu))!  u\left( \frac{p-1}{2} \right)^{c_p^{(k)}} }{p!^{c_p^{(k)}}c_p^{(k)}! (2n-N(\mu) -pc_p^{(k)} )! } \pmod{p}.
\end{align*}
Note that every summand in the first line is constant over a core configuration $\mu$. We now appeal to Wilson's theorem, that for any prime $p$ we have $(p-1)! \equiv -1 \pmod{p}$. This lets us write 
$$\frac{1}{p!^k} \equiv \frac{1}{(-1)^k p^k} \pmod{p}, $$
which completes the proof.
\end{proof}
This decomposition essentially says that the contribution from $(c_3,\ldots,c_{p-2})$ falls into one of a fixed number of cases, so that we can isolate the contribution of $c_1,c_p$. Also note that since we specify $n,k$ at the beginning, the $\mu\in \mathfrak{D}_p$ which contribute at a single step are constant across residue classes of $n-k \pmod{ \frac{p-1}{2}}$. For the simplest possible case $p=5$, the only possible core configurations are $c_3 = 0, 1$. Then, noting that coincidentally $u(0) \equiv u(1)\equiv u(2) \equiv 1 \pmod{5}$, our lemma gives 
\begin{align*}
s(n,k) \equiv \begin{cases}   
 \dfrac{(2n)!}{  (-1)^{ \frac{n-k}{2} } \left(\frac{n-k}{2} \right)! \left(   2n- 5\left(\frac{n-k}{2} \right) \right)!}, & n-k \equiv 0 \pmod{2} \vspace{4mm} \\
 \dfrac{(2n)(2n-1)(2n-2)}{ 3! }  \times  \dfrac{(2n-3)!}{  (-1)^{ \frac{n-k-1}{2} } \left(\frac{n-k-1}{2} \right)! \left(   2n-3- 5\left(\frac{n-k-1}{2} \right) \right)!},  & n-k \equiv 1 \pmod{2}.
\end{cases}
\end{align*}
which is exactly \cite[Equation (17)]{Scherer}.

\section{Final steps}
We now take a detour through the theory of the symmetric group, which forms the last link in our proof. Given the symmetric group on $n$ letters $\mathfrak{S}_n$ and a fixed prime $p$, let $X_n^{k} \subset \mathfrak{S}_n$ denote the elements formed of $k$ $p$-cycles and $n-pk$ one-cycles. Then $$|X_n^k| = \frac{n!}{k!(n-pk)!p^k} $$ and 
$$X_n := \bigcup_{k=0}^{  \left \lfloor{  \frac{n}{p}  } \right\rfloor } X_n^{k} $$
consists of all the elements in $\mathfrak{S}_n$ of order $p$. We then appeal to an old theorem of Frobenius \cite{Frobenius}.

\begin{thm}\label{frob}
Let $G$ be a finite group with $m$ dividing $|G|$. Then $m$ divides the number of solutions in $G$ to $x^m =1$. 
\end{thm}
Applying this to $\mathfrak{S}_n$, with $p\leq n$, gives $$ |X_n| = \sum_{k=0}^{  \left \lfloor{  \frac{n}{p}  } \right\rfloor } |X_n^{k}| = \sum_{k=0}^{  \left \lfloor{  \frac{n}{p}  } \right\rfloor }  \frac{n!}{k!(n-pk)!p^k}  \equiv 0 \pmod{p}.$$
The reason we are interested in such sums is that they're in \textit{almost} the same form as the inner sums in Lemma \ref{sdecomp}. We first require a change of variables argument discovered numerically; the subtlety is that due to the occurrence of various floor and ceiling functions, we must verify it for each residue class $n \pmod{p}$ separately. %Also note that the fact that we can only write such a lemma with $2n$; attempting it for arbitrary $n$ leads to very different results.

\begin{lem}\label{varchange}
Fix any prime $p$. If $0<\gamma <p$ and $0< \delta\leq \left\lfloor \frac{\gamma-1}{2} \right\rfloor$, or if $\gamma=\delta =0$, then 
$$\sum_{k=0}^{  \left \lfloor{  \frac{2n- \gamma}{p}  } \right\rfloor } |X_{2n - \gamma}^{k}| =  \sum_{  \substack{  k= \left \lceil{  \frac{n}{p}  } \right\rceil  \\   n-k \equiv \delta \mod \frac{p-1}{2} }}^n |X_{2n-\gamma}^{  \frac{2(n-k-\delta)}{  p-1 }  }|.$$
\end{lem}
\begin{proof}
Consider the right--hand side sum, and consider $k=n$ and start counting downwards. The largest $k$ that will satisfy the given congruence condition is $k = n- \delta $, for which $ \frac{2(n-k-\delta)}{  p-1 } = 0.$ As $k$ keeps decreasing by multiples of $ \frac{p-1}{2}$, $\frac{2(n-k-\delta)}{  p-1 } $ will increase to $1$, then $2$, and so on. Therefore, the right hand side will sum over $\{X_{2n-\gamma}^0,X_{2n-\gamma}^1,\ldots\}$, and we just have to verify that it halts at the correct step. Note that trivially, we have $|X|^k_{n} = 0$ when $k > \left\lfloor \frac{n}{p}\right\rfloor$, since there are no elements of $\mathfrak{S}_n$ with that many $p$-cycles, so that it's fine to sum past $X_{2n-\gamma}^{\left \lfloor{  \frac{2n- \gamma}{p}  } \right\rfloor}$ -- we just need to make sure that the smallest value of $k$ on the right hand side has summed over this term. Therefore, our equality is equivalent to instead showing the inequality
\begin{equation}\label{ineq}
\frac{2\left(n- \left \lceil{  \frac{n}{p}  } \right\rceil   -\delta \right)}{  p-1 } -  \left \lfloor{  \frac{2n- \gamma}{p}  } \right\rfloor \geq 0
\end{equation}
for the given values of $\gamma,\delta$. We will do this with a rather annoying verification based on the residue class of $n \pmod{p}$.

\textbf{Case 1: $\gamma=\delta = 0$}.
\begin{itemize}
\item \textbf{Subcase 1a: $n \equiv 0 \pmod{p}$}. Then $n = pl$ for some $l$, and \eqref{ineq} reduces to
$$\frac{2 (pl-l)}{p-1}- 2l =0. $$
\item \textbf{Subcase 1b: $n \equiv i \pmod{p}, 0<i\leq \frac{p-1}{2}$}. Then $n = pl+i$ for some $l$, and \eqref{ineq} reduces to $$\frac{2( pl+i -(l+1) )}{p-1} - 2l = \frac{2(i-1)}{p-1} \geq 0.$$
\item \textbf{Subcase 1c: $n \equiv i \pmod{p}, \frac{p+1}{2}\leq i\leq p-1$}. Then $n = pl+i$ for some $l$, and \eqref{ineq} reduces to $$\frac{2( pl+i -(l+1) )}{p-1} - (2l+1) = \frac{2(i-1) -(p-1 ) }{p-1} \geq 0.$$
\end{itemize}

\textbf{Case 2: $\gamma>0$}.
\begin{itemize}
\item \textbf{Subcase 2a: $n \equiv 0 \pmod{p}$}. Then $n = pl$ for some $l$, and \eqref{ineq} reduces to
$$\frac{2 (pl-l-\delta)}{p-1}- (2l-1) = 1 - \frac{2\delta}{p-1}. $$
However, $\gamma\leq p-1$ so \begin{equation}\label{delineq}\delta\leq \left\lfloor \frac{\gamma-1}{2} \right\rfloor\leq \frac{p-1}{2},\end{equation} and we're done.
\item \textbf{Subcase 2b: $n \equiv i \pmod{p}, 0<i < \frac{\gamma}{2}$}. Then $n = pl+i$ for some $l$, and \eqref{ineq} reduces to $$\frac{2( pl+i -(l+1)-\delta )}{p-1} - (2l-1) = 1+ \frac{2(  i-1-\delta )}{p-1} = \frac{2}{p-1}\left( (i-1) + \left( \frac{p-1}{2} -\delta \right)  \right) .$$
However, in this range we know $i \geq 1$ and $\delta \leq \frac{p-1}{2}$ (from \eqref{delineq}), so the whole line is $\geq 0$.
\item \textbf{Subcase 2c: $n \equiv i \pmod{p}, \frac{\gamma}{2} \leq i\leq \frac{p+\gamma-1}{2}$}. Then $n = pl+i$ for some $l$, and \eqref{ineq} reduces to $$\frac{2( pl+i -(l+1)-\delta )}{p-1} - 2l = \frac{2(i-1 -\delta) }{p-1} .$$
Now we again see the reason for the restriction $\delta \leq  \left\lfloor \frac{\gamma-1}{2} \right\rfloor$; this is precisely equivalent to $i-1-\delta \geq 0$ while $i\geq \frac{\gamma}{2}$.
\item \textbf{Subcase 2d: $n \equiv i \pmod{p}, \frac{p+\gamma-1}{2} < i\leq p-1$}. Then $n = pl+i$ for some $l$, and \eqref{ineq} reduces to $$\frac{2( pl+i -(l+1)-\delta )}{p-1} - (2l+1).$$Under the mapping $i \mapsto i-\frac{p-1}{2}$, this is exactly equivalent to the previous case, and the upper limit of the range of $i$ becomes $i \leq \frac{p-1}{2}-1 \leq \frac{p +\gamma-1}{2}$, since $\gamma>0$. Therefore, an application of Case (2c) completes the proof.
\end{itemize}

\end{proof}

For the $p=5$ case, Scherer \cite[Lemma 15]{Scherer} implicitly required the $\gamma=\delta=0$ and $(\delta,\gamma)=(1,3)$ cases of this lemma. The reason for this lemma is that we can show the following, rather technical, congruence.
\begin{cor}
Fix any prime $p$. If $0<\gamma <p$ and $0\leq \delta\leq \left\lfloor \frac{\gamma-1}{2} \right\rfloor$, or if $\gamma=\delta =0$, we have
$$\sum_{  \substack{  k= \left \lceil{  \frac{n}{p}  } \right\rceil  \\   n-k \equiv \delta \mod \frac{p-1}{2} }}^n \frac{\left(2n-\gamma\right)!}{ \left(\frac{2(n-k-\delta)}{  p-1 }\right)!  \left( 2n -\gamma - p\cdot \frac{2(n-k-\delta)}{  p-1 }  \right)! p^{\frac{2(n-k-\delta)}{  p-1 }} } \equiv 0 \pmod{p}.$$
\end{cor}
\begin{proof}
A straightforward application of Lemma \ref{varchange} followed by Theorem \ref{frob}.
\end{proof}

We've now built up enough machinery to inductively prove the main theorem. Throughout the rest of this section, denote 
$$\Pi(p):=   3^27^2\cdots (2p-3)^2 = \prod_{i=1}^{\frac{p-1}{2}}(4i-1)^2.$$
%denote some variant of the $p$-adic factorial.

The last key ingredient we require is a higher order analog of \cite[Lemma 14]{Scherer}: for $n\geq 3$,
$$\sum_{ \substack{ k=1 \\ k \text{\medspace even}} }^n  2^{n-k} s(n,k) \equiv \sum_{ \substack{ k=1 \\ k \text{\medspace odd}} }^n  2^{n-k} s(n,k)  \equiv 0 \pmod{5}.$$ %TODO check the floor functions in the sum limits
Note that when $p=5$, $-2^{\frac{p-1}{2}}  \Pi(p) \equiv 1 \pmod{5}$ and we perfectly recover the above result.
\begin{lem}\label{rowsum}
For a fixed residue class $l \mod{ \frac{p-1}{2}}$, we have the vanishing $$\sum_{\substack{ k=1 \\ k \equiv l \mod{  \frac{p-1}{2} } }}^n 2^{n-k}s(n,k)  \left(-2^{\frac{p-1}{2}}  \Pi(p)\right)^{ \frac{2(k-l)}{p-1} }   \equiv 0 \pmod{p}. $$
\end{lem}
\begin{proof}
We insert the decomposition \eqref{sdecomp} of $s(n,k)$ and switching the order of summation, while noting that since $k \equiv l$ that we only sum over the core configurations with $\frac{N(\mu)-\ell(\mu)}{2} \equiv n-k \equiv n-l \pmod{ \frac{p-1}{2} }$. Then we have
\begin{align}
\sum_{\substack{ k=1 \\ k \equiv l \mod{  \frac{p-1}{2} } }}^n &2^{n-k}s(n,k)  (-2^{\frac{p-1}{2}} \Pi(p))^{ \frac{2(k-l)}{p-1} } = 
 \sum_{ \substack{ \mu \in \mathfrak{D}_p   \\ \frac{N(\mu)-\ell(\mu)}{2}\equiv n-l  \mod{\frac{p-1}{2}}    }} (2n)(2n-1) \cdots (2n - N(\mu)+1) N'_\mu \nonumber \\
&  \times \sum_{\substack{ k=1 \\ k \equiv l \mod{  \frac{p-1}{2} } }}^n 2^{n-k} (-2^{\frac{p-1}{2}} \Pi(p))^{ \frac{2(k-l)}{p-1} } \frac{ (2n-N(\mu))!  u\left( \frac{p-1}{2} \right)^{c_p^{(k)}} }{ (-1)^{c_p^{(k)}} p^{c_p^{(k)}}c_p^{(k)}! (2n-N(\mu) -pc_p^{(k)} )! } \label{inner},
\end{align}
where as before
$$N'_\mu := \prod_{i=1}^{\frac{p-3}{2}} \frac{u(i)^{c_{2i+1}}}{ (2i+1)!^{c_{2i+1}} c_{2i+1}!}$$
and 
$$c_p^{(k)} := \frac{n-k - \left( \frac{N(\mu)-\ell(\mu)}{2} \right)}{  \frac{p-1}{2} }.$$ 
We will now show that each of the individual inner sums \eqref{inner} vanish. For what follows, note that by \eqref{uendpoint} we have $$u\left({\frac{p-1}{2}}\right) \equiv   3^27^2\cdots (2p-3)^2 = \Pi(p)  \pmod{p}.$$
Note that over the residue class $k\equiv l \mod{ \frac{p-1}{2}}$, we have $k = l+ m \cdot \frac{p-1}{2}$ for some integer $m$, so that $c_p^{(k)} = c_p^{(l)}-m$. Then 
\begin{align*}
2^{n-k}(-2^{\frac{p-1}{2}} \Pi(p))^{ \frac{2(k-l)}{p-1} }  &u\left( \frac{p-1}{2} \right)^{c_p^{(k)}} (-1)^{c_p^{(k)}} \\
&= 2^{n-l -m \cdot \frac{p-1}{2}}\left( -2^{\frac{p-1}{2}}\Pi(p)\right)^m \Pi(p)^{ c_p^{(l)}-m }(-1)^{ c_p^{(l)}-m } \\
&= 2^{n-l} \Pi(p)^{ c_p^{(l)}}(-1)^{ c_p^{(l)}} 
\end{align*}
is a \textit{constant} which only depends on $l$. The key insight is that as a constant function of $l$, this whole term can be factored out of the inner sum, so that line \eqref{inner} reduces to
\begin{equation}\label{inner2}
2^{n-l} \Pi(p)^{ c_p^{(l)}}(-1)^{ c_p^{(l)}}   \sum_{\substack{ k=1 \\ k \equiv l \mod{  \frac{p-1}{2} } }}^n  \frac{ (2n-N(\mu))! }{  p^{c_p^{(k)}}c_p^{(k)}! (2n-N(\mu) -pc_p^{(k)} )! }.
\end{equation}
Now remember that we originally reduced our problem to sum over core configurations satisfying $N(\mu)<p$. Furthermore, if $N(\mu)=0$ this forces $\ell(\mu)=0$. If $0<N(\mu)<p$, we must necessarily have $\ell(\mu) >0$ so that $ \frac{N(\mu)-\ell(\mu)}{2}\leq  \frac{N(\mu)-1}{2} $, which we can strengthen to $$ \frac{N(\mu)-\ell(\mu)}{2}\leq \left\lfloor \frac{N(\mu)-1}{2} \right \rfloor$$
since we know that the left--hand side is always an integer.

Note that by multiplying \eqref{ceq1} by $p$ and comparing coefficients of $c_i$ we trivially have $2n \leq 2pk$, i.e. we have $k \geq \frac{n}{p} $, which by the integrality of $k$ is equivalent to $k \geq \left\lceil \frac{n}{p} \right\rceil $. Therefore, the smallest value of $k$ which contributes to the sum is $k = \left\lceil \frac{n}{p} \right\rceil$. Hence, recalling the definition of $c_p^{(k)}$, we have satisfied the conditions of the group theoretic Lemma \ref{varchange}, and can directly apply it with $\gamma =N(\mu) $ and $\delta =  \frac{N(\mu)-\ell(\mu)}{2}$. Therefore, line \eqref{inner2} \textit{simultaneously vanishes} for every relevant core configuration.
\end{proof}

And now, many pages later, we are finally in a position to strengthen Romik's periodicity conjecture! Furthermore, note that by Euler's criterion and the law of quadratic reciprocity \cite[Chapter 9]{Apostol} $$2^{\frac{p-1}{2}} \equiv \left(\frac{2}{p}\right) = \begin{cases} 1 & p\equiv \pm 1 \pmod{8}, \\-1 & p\equiv \pm 3 \pmod{8},  \end{cases}$$
where $ \left(\frac{2}{p}\right) $ is a Legendre symbol. Therefore, we understand exactly how the $2^{\frac{p-1}{2}}$ factor behaves.
\begin{thm}
Consider a prime $p \equiv 1 \mod4$. Then for $n\geq \frac{p+1}{2}$,
\begin{equation}\label{drecur}
d\left( n+ \frac{p-1}{2} \right) \equiv -2^{\frac{p-1}{2}}  \Pi(p) d(n)  \pmod{p} . 
\end{equation}
\end{thm}
\begin{proof}
We can now proceed by induction, and assume that the result holds for $d(l),\frac{p+1}{2}\leq l<n$. Recall that the recursive definition of $d(n)$ was  %TODO check if p-1 < l or p+1 <l
$$d(n) = v(n) - \sum_{k=1}^{n-1}  2^{n-k}s(n,k)d(k).$$
By Lemma \ref{ulem}, $v(n)\equiv 0$ for $n \geq \frac{p+1}{2}$ so that we can safely neglect this term. Now for a given $k$, we write it in its base $\frac{p-1}{2}$ expansion as $k = l+ \frac{2(k-l)}{p-1} \cdot \frac{p-1}{2}$, with $l<\frac{p-1}{2}$. Then for $n\geq \frac{p+1}{2}$, we again write the base $\frac{p-1}{2}$ expansion $n = n_0 + n_1 \frac{p-1}{2}$, with $n_0 < \frac{p-1}{2}$, so that by iterating the inductive hypothesis
\begin{align*}
d(n) &\equiv -\sum_{k=1}^{n-1}  2^{n-k}s(n,k)d(k)  \pmod{p}\\
&\equiv -\sum_{l=0}^{\frac{p-3}{2}} \sum_{  \substack{k =1 \\ k \equiv l \pmod{ \frac{p-1}{2} }} }^{n-1}2^{n-k} s(n,k) d(l)\left(-2^{\frac{p-1}{2}}  \Pi(p)\right)^{ \frac{2(k-l)}{p-1} }  \pmod{p}\\
&= s(n,n)d(n_0) \left(-2^{\frac{p-1}{2}}  \Pi(p)\right)^{n_1 }    -\sum_{l=0}^{\frac{p-3}{2}}  d(l) \sum_{  \substack{k =1 \\ k \equiv l \mod{ \frac{p-1}{2} }} }^{n}2^{n-k} s(n,k)\left(-2^{\frac{p-1}{2}}  \Pi(p)\right)^{ \frac{2(k-l)}{p-1} } \pmod{p}.
\end{align*}
Now each summation over $k$ in a fixed residue class vanishes by a direct application of Lemma \ref{rowsum} and $s(n,n)=1$, so that 
$$d(n) \equiv  d(n_0) \left(-2^{\frac{p-1}{2}}  \Pi(p)\right)^{n_1 }, $$
which completes the induction since this is equivalent to our given recursion.
%TODO why is $s(n,n)=1$ always
%Now taking the base $\frac{p-1}{2}$ decomposition of $k = l+ m\cdot \frac{p-1}{2}$, where $l<\frac{p-1}{2}$, and iterating the inductive hypothesis yields
%$$d(k) \equiv \left(   -2^{\frac{p-1}{2}} \Pi(p) \right)^m d(l). $$
\end{proof}

\begin{cor}
Consider a prime $p \equiv 1 \mod4$. Then
$$d\left( n+ \frac{(p-1)^2}{2} \right) \equiv d(n) \pmod{p}. $$
\end{cor}
\begin{proof}
Note that the recurrence \eqref{drecur} is of the form $d\left( n+ \frac{p-1}{2} \right) \equiv Cd(n)  $, for some nonzero constant $C$. Iterating this $p-1$ times gives $d\left( n+ \frac{(p-1)^2}{2} \right) \equiv C^{p-1} d(n) \pmod{p}$. An appeal to Fermat's Little Theorem gives $C^{p-1}\equiv 1 \pmod{p}$ for any nonzero $C$, and completes the proof.
\end{proof}
Note that $\frac{(p-1)^2}{2}$ may not be the minimal period of $d(n)$; a fine understanding of the order of $ \prod_{i=1}^{\frac{p-1}{2}} (4i-1)^2$ is elusive.

\section{Extensions}
After  extensive numerical investigation, it appears that we should be able to lift our result to arbitrary powers of a prime. Additionally, there appears to be a similar result to our main theorem, with period $\frac{p-1}{4}$ instead. 
\begin{conj}
We can formulate the following conjectures:
\begin{enumerate}
\item Let $p$ be any prime. Then for any positive $k$, there exists an $n_k$ such that for all $n>n_k$, we have $u(n) \equiv v(n) \equiv 0 \pmod{p^k}$.
\item Consider a prime $p\equiv 1 \pmod{4}$. Then there exists a constant $C_p$ such that 
$$d\left(n+ \frac{p-1}{4}\right)\equiv C_pd(n) \pmod{p}. $$
Some first examples are $d(n+3) \equiv 4 d(n) \pmod{13}$ and $d(n+4) \equiv -2d(n) \pmod{17}$, which can be verified by computing the first $p$ terms and appealing to our Theorem \ref{mainthm}. Numerically, we conjecture $$C_p = \pm \prod_{i=1}^{ \frac{p-1}{2}}(4i-1),$$
with the sign being determined by the Legendre symbol $\left( \frac{2}{p}\right)$.
\item Consider a prime $p \equiv 1 \pmod{4}$ or $p =2$. Then for any positive $k$, there exists a period $D_{p,k}$ such that 
$$d\left(n+D_{p,k}\right)\equiv d(n) \pmod{p^k},$$
which may also support finer relations such as our Theorem \ref{mainthm}. For example, we conjecture $d(n+2^k)\equiv d(n) \pmod{2^{k+1}}$.
\item Consider a prime $p \equiv 3 \pmod{4}$. Then for any positive $k$, there exists an $n_k$ such that for all $n>n_k$, we have $d(n) \equiv 0 \pmod{p^k}$.
\end{enumerate}
\end{conj}

The reduction of $u(n)$ and $v(n)$ is heavily connected to recent work on \textit{hypergeometric supercongruences}, such as that of Victor J. W. Guo, which give binomial-type congruences modulo higher powers of a prime. Because these coefficients are defined in terms of a hypergeometric recurrence, they may be amenable to WZ style proofs. Alternatively, we may be able to recursively use the vanishing of $u(n) \mod{p^k}$ in connection with the vanishing of a binomial coefficient mod $p$ to give vanishing of $u(n) \mod{p^{k+1}}$. Another approach is to use high order analogs of Lucas's Theorem, which are however more unwieldy.

%TODO explain the k = n/5 restriction and incorporate into s_n,k decomp lemma too
%TODO effective bound on the rolen paper?? ask larry
%TODO p adic continuity of semi factorial

%}

\section{Acknowledgements}
Many thanks to Larry Washington, Christophe Vignat, Lin Jiu, and Karl Dilcher, for chatting over coffee, emailing me back at 2 am, and contributing endless blackboard space as we discussed this. The bulk of this work was completed during an idyllic summer in Halifax, and I'd like to again thank Karl Dilcher for the invitation. %I'd also like to thank Chris, Brodie, and Pascal for easing the creative process along with many backyard barbeques, and to Steph, Margaret, and Dana -- I wish we had more time to add to the Wall of Dumb!

\end{document}